\documentclass{amsart}
\usepackage{graphicx}
\usepackage{epstopdf}
\usepackage[mathscr]{eucal}

\theoremstyle{plain}
\newtheorem{prop}{Proposition}[section]
\newtheorem{coro}[prop]{Corollary}

\newtheorem{lemm}[prop]{Lemma}

\newtheorem{thm}[prop]{Theorem}

\theoremstyle{definition}
\newtheorem{defn}[prop]{Definition}

\newtheorem{rem}[prop]{Remark}

\def\mcg#1;#2{\Gamma_{#1,#2}}
\def\fg#1;#2{\Pi_{#1,#2}}
\def\tb#1;#2{\mathscr{K}_{\frac{#1}{#2}}}

\begin{document}

\title[The Jones polynomial of rational links]
{the Jones polynomial of rational links}


\author{Khaled Qazaqzeh}
\address{Department of Mathematics\\ Faculty of Science \\ Kuwait University\\
P. O. Box 5969\\ Safat-13060, Kuwait, State of Kuwait}
\email{khaled@sci.kuniv.edu.kw}

\author{Moh'd Yasein}
\address{Department of Mathematics\\ The Hashemite University\\ Zarqa, Jordan}
\email{myasein@hu.edu.jo}
\urladdr{http://staff.hu.edu.jo/myasein}

\author{Majdoleen  Abu-Qamar}
\address{Department of Mathematics\\ Yarmouk University\\
Irbid, Jordan, 21163}
\email{mjabuqamar@yahoo.com}


\date{21/05/2014}

\begin{abstract}
We give an explicit formula for the Jones polynomial of any rational link in terms
of the denominators of the canonical continued fraction of the slope of the given
rational link.
\end{abstract}

\maketitle


\section{Rational links and continued fraction}

The class of rational links have been the core of many studies since
they have been classified by Schubert in \cite{S} in terms of a
rational number called the slope. Many people since then have
studied different polynomial invariants of rational links and knots.
For example, the authors of \cite{F, M}  give an explicit formula
for the Conway (Alexander) polynomial invariant of rational links
independently. Moreover, the authors of
\cite{DS,K,LLS,LM,LZ,N1,N2,St} have studied the Jones polynomial of
rational links either directly or indirectly through studying
another polynomial invariant that reduces to the Jones polynomial
after some special normalization using different techniques.


In this paper, we give an explicit formula for the Jones polynomial
of any rational link using a different approach than the one used in
the above references. Our approach uses the Kauffman bracket state
model given in \cite{Ka} and its relation to the Tutte polynomial of
the Tait graph obtained from the diagram of the given link.

 A continued fraction of the rational number $\frac{p}{q}$ is a sequence of integers $b_{1},b_{2}, \ldots, b_{n}$ such that
\[
\frac{p}{q} = b_{1}+\cfrac{1}{b_{2}+\cfrac{1}{\ldots +\cfrac{1}{b_{n}}}}.
\]

This continued fraction of the rational number $\frac{p}{q}$ will be abbreviated by
$[b_{1}, b_{2}, \ldots, b_{n}]$. The integers $b_{i}$ are called the denominators
of the continued fraction of the rational number $\frac{p}{q}$.

Each rational link is characterized by a rational number called the slope $\frac{p}{q}$ of a pair of relatively prime integers $p,q$ with $|\frac{p}{q}| \geq 1$ and $q > 0$ by the following theorem due to Schubert \cite{S}.

\begin{thm}
Two rational links $L_{\frac{p}{q}}$ and $L_{\frac{p'}{q'}}$ are
equivalent if and only if
\begin{align*}
p&=p',
\\ \text{and }
q^{\pm1}&\equiv \pm q'(\bmod p).
\end{align*}
\end{thm}

A diagram of a rational link can be constructed from the denominators of any continued fraction of its slope by closing the 4-braid $\sigma_{1}^{b_{1}}\sigma_{2}^{-b_{2}}\sigma_{1}^{b_{3}}\ldots$ in the manner shown in figure \ref{figure1}, where $\sigma_{1}, \sigma_{2}$ are shown in figure \ref{figure2} and the multiplication is defined by concatenating from left to right. It is well known that for odd numerator $p$ this diagram represents a knot and for even numerator $p$ it represents a two component link. 

\begin{figure}[h]
 \centering
        \includegraphics[width=8cm,height=4cm]{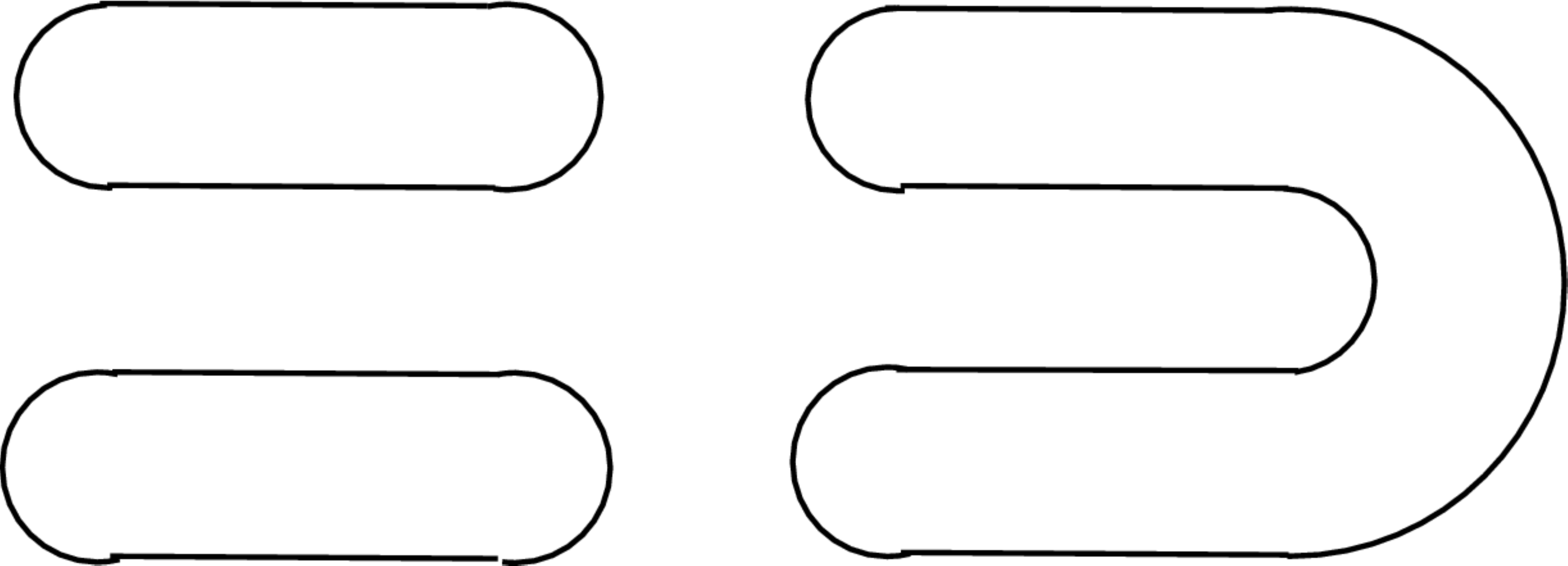}
    \caption{The odd and the even closure of the 4-braid respectively}
\label{figure1}
\end{figure}

 \begin{figure}[h]
 \centering
        \includegraphics[width=10cm,height=4cm]{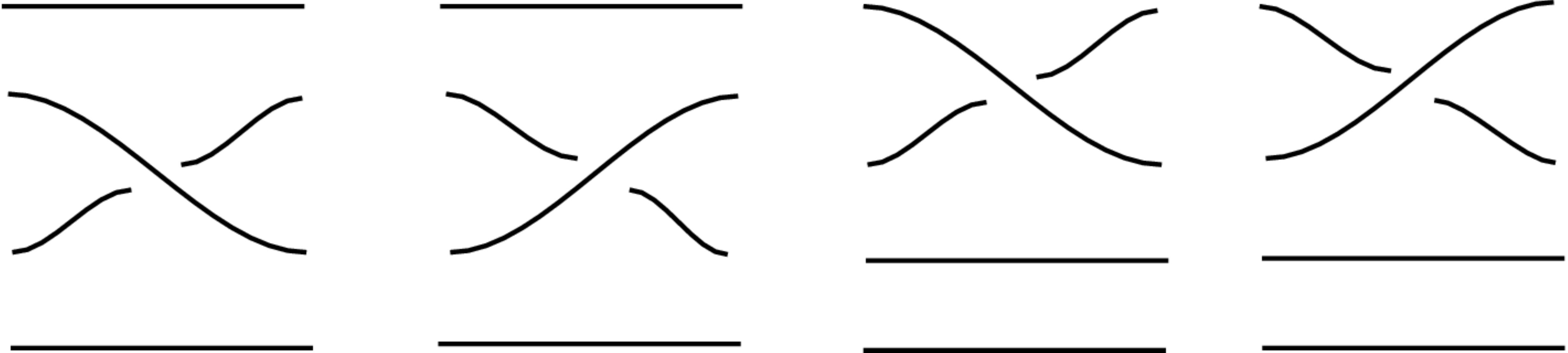}
    \caption{The 4-braids $\sigma_{1}, \sigma_{1}^{-1}, \sigma_{2},$ and $\sigma_{2}^{-1}$ respectively}
\label{figure2}
\end{figure}


It is sufficient to consider the case when the number of denominators of the continued fraction $n$ is odd and $b_{i}\geq 1$ for $ i = 1,2,\ldots n$ as a result of the following lemma.

\begin{lemm}
There exists a unique continued fraction of $\frac{p}{q} > 1$ of positive integers with $n$ odd and $b_{i}\geq 1$ for $ i = 1,2,\ldots n$.
\end{lemm}

\begin{proof}
We start with the rational number $\frac{p}{q} > 1$ such that $\gcd(p,q)=1$ and $ p> q > 0$. Thus we can apply the Euclidean algorithm to get
\begin{align*}
p&=qb_{1}+q_{1}, & 0 < q_{1} < q\\
q&=q_{1}b_{2}+q_{2},&0 < q_{2} <q_{1}\\
q_{1}&=q_{2}b_{3}+q_{3},&0 < q_{3} <q_{2}\\
& \ \ \vdots\\
q_{n-3}&=q_{n-2}b_{n-1}+q_{n-1},&0 < q_{n-1} <q_{n-2}\\
q_{n-2}&=q_{n-1}b_{n}.\\
\end{align*}

Now we have
\begin{equation*}
\frac{p}{q}= \frac{qb_{1}+q_{1}}{q} = b_{1} + \frac{1}{\frac{q}{{q_{1}}}}= b_{1} + \frac{1}{\frac{q_{1}b_{2}+q_{2}}{{q_{1}}}} = b_{1} + \frac{1}{b_{2} + \frac{q_{1}}{q_{2}}} = \dots =  b_{1}+\cfrac{1}{b_{2}+\cfrac{1}{b_{3}+\cfrac{1}{\ddots +\cfrac{1}{b_{n}}}}}.
\end{equation*}

In this way, we get a continued fraction $[b_1,b_{2},\ldots,b_n]$ of
$\frac{p}{q}$ with $b_{n} \geq 2$ since $q_{n - 1} < q_{n-2}$. Now
if $n$ is even then $[b_1,b_{2},\ldots,b_n -1, 1]$ is the continued
fraction with odd number of denominators. Finally, the uniqueness
follows from applying the Euclidean algorithm at every step.
\end{proof}

\begin{defn}
The unique continued fraction obtained using the above lemma will be called the canonical continued fraction of $\frac{p}{q}$ and the diagram obtained from the canonical continued fraction will be called the canonical diagram of the rational link whose slope is $\frac{p}{q}$. It is easy to see that the canonical diagram is alternating.
\end{defn}

\begin{rem}
The motivation of the above definition and lemma is the work of the authors in \cite[Section.\,2]{KL} for rational tangles.
\end{rem}

\begin{rem}
Most of the material of this section appears in \cite{QDQ} with the same title and we include it in here to make this paper more self-contained.
\end{rem}

\section{The Jones polynomial}

%



The Jones polynomial is an invariant of links that was first defined by V. Jones in \cite{J}. It is a Laurent polynomial in one indeterminant defined on the set of oriented links. There are many approaches to define this invariant, but we choose the approach that will serve our purposes in this paper.

The Jones polynomial of a given link can be computed using the Tutte polynomial of the associated Tait's graph of the given link diagram. In this paper, we restrict our work to alternating link diagrams. Therefore, the associated Tait's graph will be a planar graph without signs.
Now we recall the definition of the Tutte polynomial of graphs and for further details and more basic reference about this polynomial see \cite{B}. 

\begin{defn}\label{def of tutte}
The Tutte polynomial $\chi(G;x,y)\in \mathbb Z[x,y]$ of a graph $G$ is defined as follows:
\begin{enumerate}
    \item If the graph $G$ consists only of the vertex $v$, then $\chi
    (v)=1$.
    \item If the graph $G$ consists only of the edge $e$, then$\chi
    (e)=x$.
    \item If the graph $G$ consists only of the loop $l$, then$\chi
    (l)=y$.
    \item\label{4} If $G_1*G_2$ denotes a connected graph consists of two graphs
    $G_1$ and $G_2$ having just one vertex in common, then $\chi (G_1*G_2)=\chi(G_1) \chi(G_2)$.
    \item If $G_1\sqcup G_2$ is the disjoint union of the two graphs $G_{1}$ and $G_{2}$, then $\chi (G_1\sqcup G_2)=\chi(G_1) \chi(G_2)$.
    \item\label{6} If $e$ is an edge which is neither a loop nor a bridge of the graph $G$, then
    $\chi (G)=\chi (G-e)+\chi(G\backslash e)$ where $G-e$ is the graph obtained be deleting the edge $e$
    in $G$ and $G\backslash e$ is the graph obtained by contracting the edge $e$ in $G$.
\end{enumerate}
In a graph $G$ a bridge is an edge whose removal increases the
number of components of $G$ and a loop is an edge which has the same
vertex as its endpoints.

\end{defn}



The way to construct the Tait's graph of a given alternating link
diagram is by using the checkerboard coloring, that is we color the
regions of the diagram in $\mathbb{R}^{2}$ into two colors black and
white such that regions which share an arc have different colors. We
then place a vertex in each black region and associate an edge to
each crossing of the link that connects two vertices to obtain the
graph $G$. By interchanging black regions with white regions,
we obtain the dual graph of $G$. 

We quote the following lemma that first appeared in \cite{K1}.
\begin{lemm}
If the outside region is white, then the Tait's graph of the canonical rational link diagram takes the form of graph given in figure \ref{figure4}.
\end{lemm}

\begin{defn}
The Tait's graph corresponding to the the canonical rational link diagram will be called the canonical Tait's graph of the given rational link.
\end{defn}

The Jones polynomial of an oriented link can be expressed via the Tutte polynomial of the Tait's graph in \cite{B} by the following theorem:
\begin{thm} \label{jones via tutte}
The Jones polynomial $V_l(t)$ of an alternating link $L$ can be obtained from the Tutte polynomial $\chi(G;x,y)$ of the assocaited Tait's graph $G$ by the following equation:
\begin{equation*}
V_L(t)=(-1)^{w}t^{\frac{a-b-3w}{4}} \chi(G;-t,-t^{-1})
\end{equation*}
where  $a$ is the number of white regions, $b$ is the number of black regions, and $w$ is the writhe of the link diagram.

\end{thm}

\begin{defn}
The writhe of a diagram of an oriented link is the number of the crossings of type $L_{+}$ minus the number of crossings of type $L_{-}$ as given in figure \ref{figure3}.
\end{defn}

 \begin{figure}[h]
    \centering
        \includegraphics[width=16cm,height=10cm]{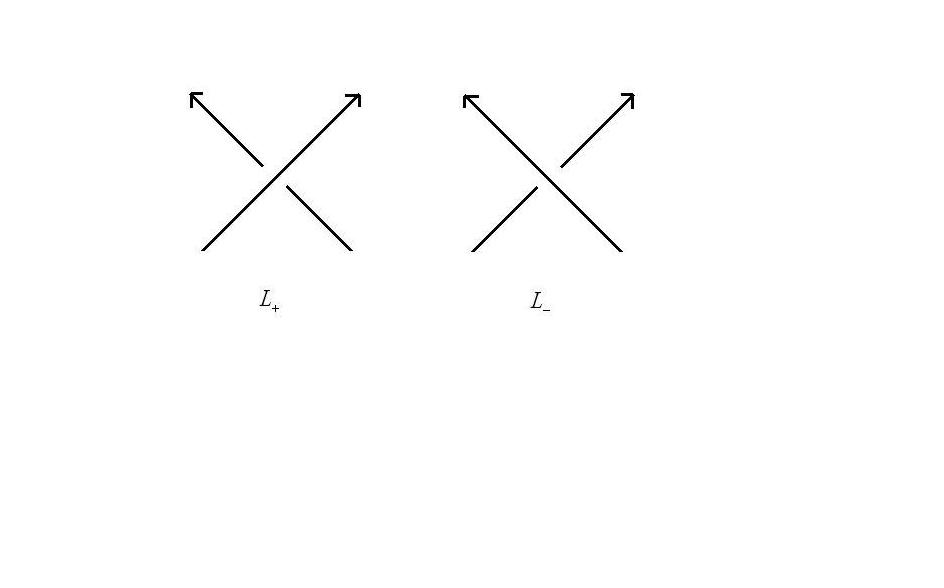}
        \vspace{-5cm}
        \caption{$L_+$, and $L_{-}$}
    \label{figure3}
\end{figure}


\section{The Tutte polynomial of the canonical Tait's Graphs}

We give a formula for the Tutte polynomial of the canonical Tait's graph of any rational link in terms of the denominators of the canonical continued fraction of the slope of the given rational link. First, we recall that a cycle graph of length $p$ is a graph with $p$ vertices and $p$ consecutive edges such that each vertex is incident to two edges.


We quote the following lemmas for later use whose proofs can be
found in any basic reference of graph theory see for example
\cite{W}.

\begin{lemm}\label{tutte of cycle}
Let $C_p$ be a cycle graph with $p$ edges then the Tutte polynomial is
\begin{equation*}
\chi (C_p)=\frac{x^p-1}{x-1}+y-1.
\end{equation*}
\end{lemm}

\begin{lemm}\label{tutte of dual}
The Tutte polynomial of the dual graph of a graph $G$ equals to the Tutte polynomial of the original graph after interchanging $x$ and $y$.
\end{lemm}



The canonical Tait's graph of any rational link is a graph as shown in figure \ref{figure4}
where $b_{i}$ denotes the number of edges that are parallel if $i$ is odd and collinear
if $i$ is even with $2+\sum\limits_{i=1}^{k} b_{2i}$ vertices and  $\sum\limits_{i=1}^{2k + 1} b_i $ edges. Let $E=\{b_2,b_4,\ldots ,b_{2k}\}$, $O=\{b_1,b_3,\ldots ,b_{2k+1}\}$, $C=\{x: x=\sum\limits_{i=l}^{m}{b_{2i-1}}, 1 \leq l,m \leq k+1\}$ and $\rho(E),\text{ }\rho(C)$ denote the power sets for the sets $E,\text{and }C $ respectively. We define $f:\rho (E)  \rightarrow \rho (C)$ by 
$f(A)=B$, where $x\in B$ iff $ x$ is one of the following forms

\begin{enumerate}
 \item If $A = \phi$, then $ x =\sum\limits_{i=1}^{2k + 1} b_i$.
 \item If $b_l,b_m\in A$ and $b_n\notin A$ for $l< n< m$, then  $x=\underset{l\leq i\leq m, b_{i} \in O}{\sum} b_i$.
 \item If $b_m\in A$ and $b_n\notin A$ for $1\leq n< m$, then $ x=\underset{1\leq i\leq m, b_{i} \in O}{\sum} b_{i}$.
 \item If $b_l\in A$ and $b_n\notin A$ for $l< n< 2k+1$, then $x=\underset{l\leq i\leq 2k+1, b_{i} \in O}{\sum}b_i$.

\end{enumerate}

\begin{figure}[h]
        \includegraphics[width=15cm,height=9cm]{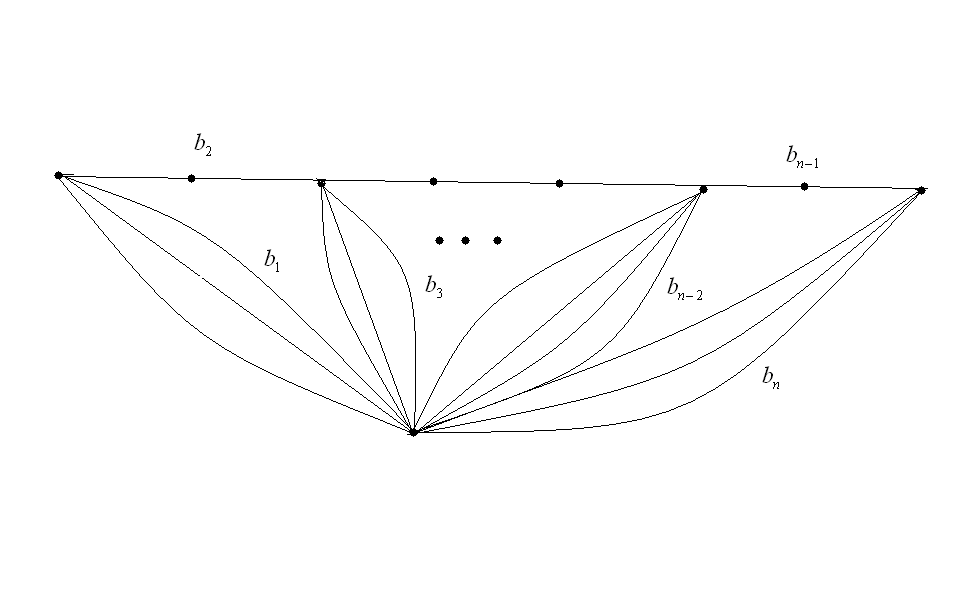}
            \vspace{-2cm}
    \caption{The Graph related to the sequence $\{b_1,\ldots,b_{n}\}$ of positive integers.}
    \label{figure4}
\end{figure}

Now, we state the main theorem of this section:
\begin{thm} \label{tutte}
The Tutte polynomial of the graph shown in figure \ref{figure4} is given by the formula
\begin{equation*}
\chi (G)=\sum_{A\subseteq E}\prod_{{b_i}\in A}(\frac{x^{b_i}-1}{x-1}) \prod_{{\alpha_i}\in f(A)}(\frac{y^{\alpha_i}-1}{y-1}+x-1).
\end{equation*}

\end{thm}

\begin{proof}
We will use induction on $k$. If $k=0$ then $G$ will be the dual graph of $C_{b_1}$ so from Lemmas \ref{tutte of dual} and \ref{tutte of cycle} we get
\[
\chi(G)=\frac{y^{b_1}-1}{y-1}+x-1.
\]

Now for $k = m$, we apply part \ref{6} of Definition \ref{def of
tutte} on one of the $b_{2k}$-edges that are collinear in figure
\ref{figure4} and use part \ref{4} of the same definition to get
\[
\chi(G)=x^{b_{2k}-1} (\frac{y^{b_{2k+1}}-1}{y-1}+x-1)\chi (G')+\chi(G''),
\]
where $G', G''$ are the canonical Tait's graphs corresponding to the canonical continued fractions $[b_1,b_2,\ldots,b_{2k-1}]$, and $ [b_1,b_2,\ldots,b_{2k-1},b_{2k}-1,b_{2k+1}]$ respectively. 
Thus repeating this process $b_{2k}-1$ times on the graph $G''$, we obtain
\begin{align*}\label{G}
\chi(G) = \ & (\frac{x^{b_{2k}}-1}{x-1})(\frac{y^{b_{2k+1}}-1}{y-1}+x-1)\chi (G')+\chi(G^{\alpha})\\
= \ & (\frac{x^{b_{2k}}-1}{x-1})(\frac{y^{b_{2k+1}}-1}{y-1}+x-1)\sum_{A\subseteq E'}\prod_{{b_i}\in A}(\frac{x^{b_i}-1}{x-1}) \prod_{{\alpha_i}\in f(A)}(\frac{y^{\alpha_i}-1}{y-1}+x-1) \\ & \hspace{5cm}   +   \sum_{A\subseteq E^\alpha}\prod_{{b_i}\in A}(\frac{x^{b_i}-1}{x-1}) \prod_{{\alpha_i}\in f(A)}(\frac{y^{\alpha_i}-1}{y-1}+x-1)\\
= \ & \sum_{A\subseteq E, b_{2k} \in A}\prod_{{b_i}\in A}(\frac{x^{b_i}-1}{x-1}) \prod_{{\alpha_i}\in f(A)}(\frac{y^{\alpha_i}-1}{y-1}+x-1) \\ & \hspace{5cm}   +  \sum_{A\subseteq E, b_{2k} \notin A}\prod_{{b_i}\in A}(\frac{x^{b_i}-1}{x-1}) \prod_{{\alpha_i}\in f(A)}(\frac{y^{\alpha_i}-1}{y-1}+x-1) \\
= \ &  \sum_{A\subseteq E}\prod_{{b_i}\in A}(\frac{x^{b_i}-1}{x-1}) \prod_{{\alpha_i}\in f(A)}(\frac{y^{\alpha_i}-1}{y-1}+x-1),
\end{align*}
where $G^\alpha$ is the canonical Tait's graphs corresponding to the canonical continued fraction $[b_1,b_2,\ldots,b_{2k-1}+b_{2k+1}]$ and the second equality follows from the induction hypothesis on $G'$ and $G^{\alpha}$.

\end{proof}

\begin{coro} \label{cor1}
The Tutte polynomial of the canonical Tait's graph that corresponds to the rational link $C(b_{1})$ in Conway's notation in \cite{C} 
is given by
\[
\chi(G;x,y)=\frac{y^{b_1}-1}{y-1}+x-1.
\]
\end{coro}
\begin{proof}
The result follows since the canonical continued fraction the rational link $C(b_{1})$ is $[b_1]$ and $E=\{\phi\}$.
\end{proof}
\begin{coro} \label{cor2}
The Tutte polynomial of the canonical Tait's graph that corresponds to the rational link $C(b_{1},b_{2})$ in Conway's notation in \cite{C} 
is given by
\[
\chi(G;x,y)=x(\frac{x^{b_{2} - 1}-1}{x-1})(\frac{y^{b_1}-1}{y-1}+x-1)+\frac{y^{b_1+1}-1}{y-1}+x-1.
\]
\end{coro}
\begin{proof}
The result follows since the canonical continued fraction the rational link $C(b_{1},b_{2})$ is $[b_1,b_{2} - 1,1]$ and $E=\{b_{2} - 1\}$.
\end{proof}
\section{Main Results}

For this section, we let $D$ be the canonical link diagram of the rational link $L$ with slope $\frac{p}{q}$ of canonical continued fraction $[b_{1},b_{2},\ldots,b_{n}]$.

We consider the case where $\frac{p}{q} \geq 1$ since the other case yields the mirror image of the link with slope $|\frac{p}{q}|$ and the relation between the Jones polynomial of a link and the Jones polynomial of its mirror image is given by the following theorem:

\begin{thm}
Suppose $K^*$ is the mirror image of a link $K$, then
\[
V_{K^*}(t)=V_{K}(t^{-1}).
\]
\end{thm}

We want to compute the number of white regions, the number of white
regions, and the writhe of the canonical diagram $D$ in terms of the
denominators of the canonical continued fraction that will be used
in the Theorem \ref{main1}.

\begin{lemm} \label{abw}
Let $G$ be the corresponding Tait's graph of the canonical diagram $D$, then
\begin{align*}
a& =k+1+\sum\limits_{i=1}^{k+1}{(b_{2i-1}-1)}= \sum\limits_{i=1}^{k+1}b_{2i - 1}.\\
b& = |V_{G}|=2+\sum\limits_{i = 1}^{k}b_{2i}.\\
\end{align*}

\end{lemm}

We associate to the canonical diagram $D$ a permutation $\sigma_{D} \in S_{3}$ on the set $\{1,2,3\}$. We define the permutation $\sigma_{D}$ in terms of the denominators of the canonical continued fraction by

\[ \left(\begin{array}{clrr}      1 & 0 & 0  \\       0  & 0 & 1 \\ 0 & 1 & 0 \end{array}\right)^{b_{1}} \left(\begin{array}{clrr} 0 & 1 & 0  \\       1  & 0 & 0 \\ 0 & 0 & 1 \end{array}\right)^{b_{2}}\ldots \left(\begin{array}{clrr} 1 & 0 & 0  \\       0  & 0 & 1 \\ 0 & 1 & 0 \end{array}\right)^{b_{2k + 1}}\left(\begin{array}{clrr} 1 \\ 2 \\ 3 \end{array}\right) = \left(\begin{array}{clrr} \sigma_{D}(1)  \\   \sigma_{D}(2) \\ \sigma_{D}(3) \end{array}\right). \]

Now the writhe of the canonical diagram $D$ depends on the permutation $\sigma_{D}$. In particular, we have four cases for the writhe and it is given by the following lemma
\begin{prop}\label{writhe}
The writhe $w$ of the canonical diagram $D$ is given recursively by
\[ w =
  \begin{cases}
  b_{1} + b_{2} + \sum\limits_{i=3}^{n}\epsilon_{i}b_{i} ,
     & \text{if $\sigma_{D} = (1)$ or $\sigma_{D} = (23)$ or $\sigma_{D} = (12)$ or $\sigma_{D} = (123)$},
     \\
 -b_{1} + b_{2} + \sum\limits_{i=3}^{n}\epsilon_{i}b_{i} , & \text{if $(\sigma_{D} = (13)$ or $\sigma_{D} = (132))$ and $b_{1}$ is even},
 \\
 -b_{1} - b_{2} + \sum\limits_{i=3}^{n}\epsilon_{i}b_{i} , & \text{if $(\sigma_{D} = (13)$ or $\sigma_{D} = (132))$ and $b_{1}$ is odd},
  \end{cases}
\]
where
\[
\epsilon_{i} =
\begin{cases}
 - \epsilon_{i-2},
     & \text{if $b_{i-1}$ is odd, $i-1$ is even and $\epsilon_{i-1} = 1$},
     \\
     - \epsilon_{i-2},
     & \text{if $b_{i-1}$ is odd, $i-1$ is odd and $\epsilon_{i-1} = -1$ },
     \\
   \epsilon_{i-2} , & \text{otherwise}.
  \end{cases}
\]
\end{prop}
\begin{proof}
We prove the case where $\sigma_{D} = (23)$. In this case, the canonical diagram $D$ will be closed as in figure \ref{figure5}. We choose the orientation in a way where the top arc always goes from right to left and if the diagram has two components then we can assume the orientation on the bottom arc goes from right to left since these two arcs will belong to different components.

The set of all crossings in the canonical diagram $D$ forms a partition of $n$ elements such that $i$-th element of this partition contains all the crossings that form $\sigma_{1}^{b_{i}}$ if $i$ is odd and $\sigma_{2}^{-b_{i}}$ if $i$ is even in the braid form. 
It is clear that crossings of the same element of the partition have
the same sign. Therefore, we have $ w = \sum\limits_{i=1}^{2k+1}
\epsilon_{i}b_{i}$. Now after we choose the orientation as above, we
obtain $ \epsilon_{1} = \epsilon_{2} = 1$. Assume that we determine
the value of $\epsilon_{i}$ for $1\leq i \leq m-1$ and we want to
determine the value of $\epsilon_{m}$. We note that the value of
$\epsilon_{m}$  depends on the parity of $b_{m-2}, b_{m-1}$ and the
values of $\epsilon_{m-2},\epsilon_{m-1}$. Therefore, we can
consider the values of $b_{m-2}, b_{m-1}$ of being 1 or 2 in the
case that $b_{i}$ is odd or even respectively for $ i = m-2, m-1$.
Now we show one case as in figure \ref{figure5} and the other cases
will be treated similarly.
\end{proof}

\begin{figure}[h]
  \centering
        \includegraphics[width=6cm,height=9cm]{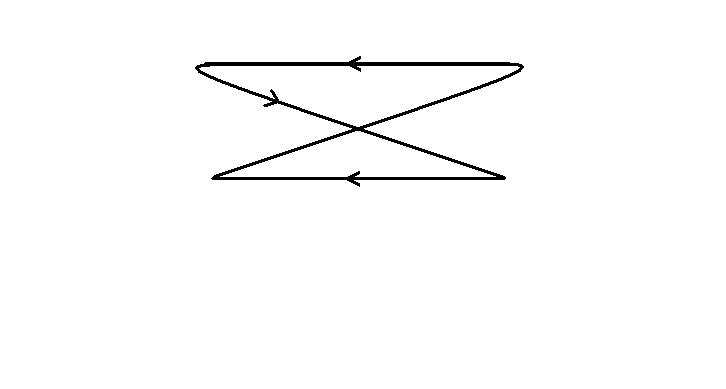}
            \includegraphics[width=8cm,height=9cm]{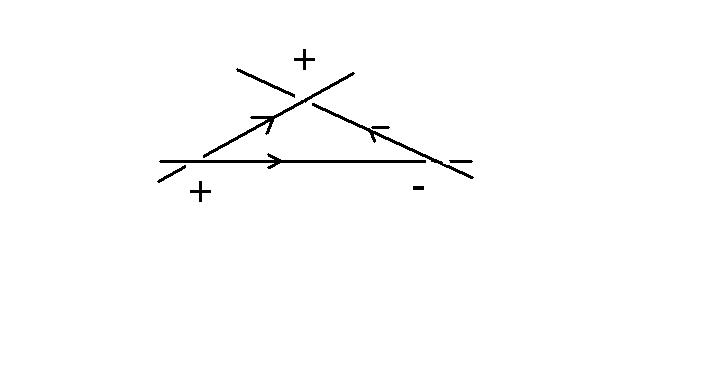}
        \vspace{-4cm}
    \caption{The case where $\sigma_{D} = (23)$ and one of the cases of $\epsilon_{i}$.}
    \label{figure5}
\end{figure}

From Theorems \ref{tutte},
\ref{jones via tutte} and Lemma \ref{abw}, we get a formula of the Jones polynomial of rational links.
\begin{thm}\label{main1}
The Jones polynomial of the rational link $L$ with the canonical
continued fraction $[b_1,b_2,\ldots,b_{2k+1}]$ is
\begin{equation} \label{main}
V_L(t)=(-1)^{w}t^{\frac{\sum\limits_{i=1}^{k+1}b_{2i-1}-(2+\sum\limits_{i=1}^{k}b_{2i})-3w}{4}}
\chi(G;-t,-t^{-1}).
\end{equation}
where $\chi(G;-t,-t^{-1})$ 
is the Tutte polynomial of the graph $G$ shown in figure
\ref{figure4} computed in Theorem \ref{tutte} and $w$ is the writhe
computed in Proposition \ref{writhe}.
\end{thm}

\begin{coro}
The determinant of the rational link $L$ with the canonical
continued fraction $[b_1,b_2,\ldots,b_{2k+1}]$ is
\begin{equation} \label{mainminor}
\det (L) = \sum_{A\subseteq E}\prod_{{b_i}\in A}b_{i} \prod_{{\alpha_i}\in f(A)}\alpha_{i}.
\end{equation}
\end{coro}

\begin{coro}
The Jones polynomial of rational link $C(b_{1})$ in Conway's notation in \cite{C} is given by
\[
V_l(t)=(-1)^{(b_1+1)}t^{\frac{-b_1+1}{2}}+\sum\limits_{i=1}^{b_1-1}(-1)^{i}(t^{-1})^{\frac{3b_1-(2i-1)}{2}}.
\]
\end{coro}
\begin{proof}
For any rational link with one denominator, we can take an
orientation such that $w = b_1$. Corollary \ref{cor1} implies
\[\chi(G;-t,-t^{-1})=\frac{(-t^{-1})^{b_1}-1}{-t^{-1}-1}-t-1
\]
We substitute in equation \ref{main} to get
    \begin{align*}
    V_L(t)&=(-1)^{b_1}t^{\frac{b_1-2-3b_1}{4}}(\frac{(-t^{-1})^{b_1}-1}{-t^{-1}-1}-t-1)\\
    &=(-1)^{b_1}(t^{-1})^{\frac{(b_1+1)}{2}}(\sum\limits_{i=1}^{b_1-1}(-1)^{b_1-i}(t^{-1})^{b_1-i}-t)\\
&=(-1)^{(b_1+1)}t^{\frac{-b_1+1}{2}}+\sum\limits_{i=1}^{b_1-1}(-1)^{i}(t^{-1})^{\frac{3b_1-(2i-1)}{2}}.
\end{align*}
\end{proof}

\begin{coro}
The Jones polynomial of the rational link $C(b_{1},b_{2})$ in Conway's notation in \cite{C} is
\[
V_l(t)=(-1)^{w}t^{\frac{b_{1} - b_{2} -3w}{4}}\left((\frac{(-t)^{b_2-1}-1}{-t-1})
(\frac{(-t^{-1})^{b_1}-1}{-t^{-1}-1}-t-1)(-t)+(\frac{(-t^{-1})^{b_1+1}-1}{-t^{-1}-1}-t-1)\right).
\]
\end{coro}
\begin{proof}
For the rational link with two denominators, we have
\[
w=\begin{cases}
{b_1+b_2}, &\text{if $b_1,b_2\equiv 1(\bmod2)$  or  $b_1\equiv 0(\bmod2), b_2\equiv 1(\bmod2)$},
\\
{-(b_1+b_2)}, &\text{if $b_1\equiv 1(\bmod2), b_2\equiv 0(\bmod2)$},
\\
{-b_1+b_2}, &\text{if $b_1,b_2\equiv 0(\bmod2)$}.
\end{cases}
\]
Corollary \ref{cor2} gives
\[\chi(G;-t,-t^{-1})=-t(\frac{(-t)^{{b_2}-1}-1}{-t-1})(\frac{(-t^{-1})^{b_1}-1}{-t^{-1}-1}-t-1)+\frac{(-t^{-1})^{b_1+1}-1}{-t^{-1}-1}-t-1.
\]
Substitute in equation \ref{main} we get
\[
V_l(t)=(-1)^{w}t^{\frac{b_{1} - b_{2} - 3w}{4}}\left((\frac{(-t)^{b_2-1}-1}{-t-1})
(\frac{(-t^{-1})^{b_1}-1}{-t^{-1}-1}-t-1)(-t)+(\frac{(-t^{-1})^{b_1+1}-1}{-t^{-1}-1}-t-1)\right).
\]
\end{proof}


\end{document}